\documentclass[12pt]{amsart}
\usepackage{txfonts}
\usepackage{amscd,amssymb,mathabx,subfigure}
\usepackage[all]{xy}
\usepackage{graphicx,calrsfs}
\usepackage[colorlinks,plainpages,backref,urlcolor=blue]{hyperref}
\usepackage{tikz}
\usetikzlibrary{decorations.markings}
\usetikzlibrary{shapes}
\usetikzlibrary{backgrounds}
\usetikzlibrary{calc}
\usepackage{mdwlist}
\usepackage{verbatim}

\newtheorem{theorem}{Theorem}[section]
\newtheorem{corollary}[theorem]{Corollary}
\newtheorem{lemma}[theorem]{Lemma}
\newtheorem{prop}[theorem]{Proposition}

\theoremstyle{definition}
\newtheorem{definition}[theorem]{Definition}
\newtheorem{example}[theorem]{Example}

\newtheorem{conjecture}[theorem]{Conjecture}

\newenvironment{romenum}
{ 

\begin{enumerate}}{\end{enumerate}}

\newcommand{\N}{\mathbb{N}}
\newcommand{\Z}{\mathbb{Z}}

\newcommand{\R}{\mathbb{R}}
\newcommand{\C}{\mathbb{C}}
\newcommand{\F}{\mathbb{F}}

\newcommand{\CP}{\mathbb{CP}}

\renewcommand{\P}{\mathbb{P}}

\renewcommand{\k}{\Bbbk}
\newcommand{\RR}{\mathcal{R}}
\newcommand{\VV}{\mathcal{V}}
\newcommand{\A}{{\mathcal{A}}}
\newcommand{\B}{{\mathcal{B}}}

\newcommand{\Uc}{{\overline{U}}}
\newcommand{\bdU}{{\partial{\Uc}}}
\newcommand{\Fc}{{\overline{F}}}
\newcommand{\bdF}{{\partial{\Fc}}}

\DeclareMathAlphabet{\pazocal}{OMS}{zplm}{m}{n}
\newcommand{\XX}{{\pazocal X}}
\newcommand{\NN}{{\pazocal{N}}}

\DeclareMathOperator{\rank}{rank}

\DeclareMathOperator{\codim}{codim}

\DeclareMathOperator{\id}{id}
\DeclareMathOperator{\ab}{{ab}}

\DeclareMathOperator{\ch}{char}

\DeclareMathOperator{\Hom}{{Hom}}

\DeclareMathOperator{\depth}{depth}

\newcommand{\bo}{{\mathbf 1}}

\newcommand{\ii}{{\mathrm i}}

\newcommand{\surj}{\twoheadrightarrow}
\newcommand{\inj}{\hookrightarrow}

\newcommand{\abs}[1]{\left| #1 \right|}

\def\set#1{{\{ #1\}}}
   
\newcommand{\norm}[1]{|\!|#1|\!|}

\definecolor{dkgreen}{RGB}{0,100,0}
\definecolor{dkbrown}{RGB}{139,69,19}

\def\dot{\mathchar"013A}  
\newcommand{\hdot}{{\raise1pt\hbox to0.35em{\Huge $\dot$}}} 

\title[Milnor fibrations of hyperplane arrangements]{%
On the topology of the Milnor fibration of \\ a hyperplane arrangement} 

\author[Alexandru~I.~Suciu]{Alexandru~I.~Suciu}
\address{Department of Mathematics,
Northeastern University,
Boston, MA 02115, USA}
\email{a.suciu@neu.edu}
\urladdr{http://www.northeastern.edu/suciu/}
\thanks{Partially supported by the Simons Foundation 
collaboration grant for mathematicians 354156}

\subjclass[2010]{Primary
32S55,  
52C35;  
Secondary
05B35,  
14C21,  
14F35, 
32S22,  
55N25, 
57M10  
}

\keywords{Hyperplane arrangement, Milnor fibration, 
boundary manifold, resonance variety, characteristic variety, 
algebraic monodromy.}

\begin{document}

\begin{abstract}
This note is mostly an expository survey, centered on the 
topology of complements of hyperplane arrangements, 
their Milnor fibrations, and their boundary structures.  
An important tool in this study is provided by the degree~$1$  
resonance and characteristic varieties of the complement, 
and their tight relationship with orbifold fibrations and 
multinets on the underlying matroid.  In favorable situations, 
this approach leads to a combinatorial formula for the first 
Betti number of the Milnor fiber and the algebraic 
monodromy.  We also produce a pair of arrangements 
for which the respective Milnor fibers have the same Betti numbers, 
yet are not homotopy equivalent: the difference is picked up 
by isolated torsion points in the higher-depth characteristic varieties. 
\end{abstract}
\maketitle

\setcounter{tocdepth}{1}
\tableofcontents

\section{Introduction}
\label{sect:intro}

\subsection{The Milnor fibration}
\label{subsec:intro1}
A construction due to J.~Milnor \cite{Mi} associates 
to each homogeneous polynomial $Q\in \C[z_0,\dots , z_d]$ 
a fiber bundle, with base space $\C^*=\C\setminus \{0\}$, 
total space the complement in $\C^{d+1}$ to the hypersurface 
$V$ given by the vanishing of $Q$, and projection map 
$Q\colon \C^{d+1} \setminus V\to \C^*$.

The Milnor fiber $F=Q^{-1}(1)$ is a Stein manifold, and thus 
has the homotopy type of a finite, $d$-dimensional CW-complex. 
The monodromy of the fibration, $h\colon F\to F$, is given by 
$h(z)=e^{2\pi i/n} z$, where $n$ is the degree of $Q$. If the 
polynomial $Q$ has an isolated singularity at the origin, then 
$F$ is homotopy equivalent to a bouquet of $d$-spheres, whose 
number can be determined by algebraic means. In general, though, 
it is a rather hard problem to compute the homology groups of the 
Milnor fiber, even in the case when $Q$ completely factors into 
distinct linear forms. 

This situation is best described by a hyperplane arrangement, that is, 
a finite collection of codi\-mension-$1$ linear subspaces in $\C^{d+1}$. 
Choosing a linear form $f_H$ with kernel $H$ for each hyperplane $H$ 
in the arrangement $\A$, we obtain a homogeneous polynomial, 
$Q=\prod_{H\in \A} f_H$, which in turn defines the Milnor fibration 
of the arrangement, and the Milnor fiber, $F=F(\A)$.  A central 
question in the subject is to determine whether $\Delta_{\A}(t)$, the 
characteristic polynomial of the algebraic monodromy $h_*\colon 
H_1(F,\C)\to H_1(F,\C)$, is determined by the intersection lattice 
of the arrangement, $L(\A)$.  We present here some recent progress 
on this and other related questions, mostly based on  \cite{Su-toul} 
and on joint work with G.~Denham \cite{DeS-plms} and 
S.~Papadima \cite{PS-betamilnor}.

\subsection{Complement and jump loci}
\label{subsec:intro2}
Let $U$ be the complement of the complexified arrangement, $\bar\A=\P(\A)$.  
It turns out that the Milnor fiber 
$F$ is a regular, cyclic $n$-fold cover of $U$, where $n=\abs{\A}$. 
The classifying homomorphism for this cover, $\delta\colon \pi_1(U)\to \Z_n$, 
takes each meridian loop around a hyperplane to $1$.  Embedding $\Z_n$ 
into $\C^*$ by sending $1$ to a primitive $n$-th root of unity, we may 
view $\delta$ as a character on $\pi_1(U)$, see \cite{CS95, Su-toul}. 
The relative position of this character with respect to the characteristic 
varieties of $U$ determines the Betti numbers of $F$, as 
well as the characteristic polynomial of the algebraic monodromy.  

Since $U$ is a smooth, quasi-projective variety, its characteristic 
varieties are finite unions of torsion-translates of algebraic subtori 
of the character group $\Hom(\pi_1(U),\C^*)$, cf.~\cite{Ar, ACM, BW}.  
Since $U$ is also a formal space, its resonance varieties (defined 
in terms of the Orlik--Solomon algebra of $\A$) coincide 
with the tangent cone at the origin to the corresponding 
characteristic varieties, cf.~\cite{CS99, Li01, DPS-duke, DP-ccm}.  
As shown by Falk and Yuzvinsky \cite{FY}, the degree $1$ resonance 
varieties may be described solely in terms of multinets on sub-arrangements 
of $\A$.  In general, though, the degree $1$ characteristic varieties 
of an arrangement may contain components which do not pass through 
the origin \cite{S02, DeS-plms}, and it is still an open problem whether 
such components are combinatorially determined.

Under simple combinatorial conditions, it is shown in \cite{PS-betamilnor} 
that the multiplicities of the factors of $\Delta_{\A}(t)$ corresponding 
to certain eigenvalues of order a power of a prime $p$ are equal 
to the `Aomoto--Betti numbers' $\beta_p(\A)$, which in turn can be  
extracted from the intersection $L(\A)$ by considering the resonance 
varieties of $U$ over a field of characteristic $p$.  When $\bar\A$ is  
an arrangement of projective lines with only double and triple points, 
this approach leads to a combinatorial formula for the algebraic monodromy. 

\subsection{Boundary structures}
\label{subsec:intro3}

Both the projectivized complement $U$ and the Milnor fiber $F$ 
are boundaryless, non-compact manifolds.   
Removing a regular neighborhood of the arrangement 
yields a compact manifold with boundary, $\Uc$, onto which $U$  
deform-retracts.  Likewise, intersecting the Milnor fiber with a ball 
in $\C^{d+1}$ centered at the origin yields a compact manifold with 
boundary, $\Fc$, onto which $F$ deform-retracts.  

We focus on the case $d=2$, when both the boundary manifold 
of the arrangement, $\bdU$, and the boundary of the Milnor fiber, 
$\bdF$, are closed, orientable, $3$-dimensional graph manifolds. 
Once again, there is a regular, cyclic $n$-fold cover $\bdF\to \bdU$, 
whose classifying map can be described in concrete terms.
Various topological invariants of these manifolds, including 
the cohomology ring and the depth $1$ characteristic variety 
of $\bdU$ \cite{CS06, CS08}, as well as the Betti numbers of 
$\bdF$ \cite{NS}, can be computed from the combinatorics of $\A$.  

In \cite{Fa93}, Falk produced a pair arrangements, $\A$ and $\A'$, 
for which the intersection lattices are non-isomorphic, but the projective 
complements, $U$ and $U'$, are homotopy equivalent.  Nevertheless, 
the boundary manifolds are not homotopy equivalent, \cite{JY93, CS08}, 
and thus the complements are not homeomorphic.  We show here that 
the respective Milnor fibers, $F$ and $F'$, as well as their boundaries, 
$\bdF$ and $\bdF'$, have the same Betti numbers, but that $F\not\simeq F'$. 
The difference between the two Milnor fibers is detected by the depth $2$ 
characteristic varieties:  $\VV_2(F)=\{\bo\}$, whereas 
 $\VV_2(F') \cong \Z_3$.

\section{Complement, boundary manifold, and Milnor fibration}
\label{sect:milnor}

\subsection{The complement of a hyperplane arrangement}
\label{subsec:hyp arr}

An {\em arrangement of hyperplanes}\/ is a finite set $\A$ of 
codimension-$1$ linear subspaces in a finite-dimensional, 
complex vector space $\C^{d+1}$.  The combinatorics of the 
arrangement is encoded in its {\em intersection lattice}, $L(\A)$,   
that is, the poset of all intersections of hyperplanes in $\A$ (also 
known as flats), ordered by reverse inclusion, and ranked by 
codimension.  Given a flat $X$, we will denote by $\A_X$ the 
sub-arrangement $\{H\in \A\mid H\supset X\}$. 

Without much loss of generality, we will assume throughout 
that the arrangement is {\em central}, that is, all the hyperplanes 
pass through the origin.   For each hyperplane $H\in \A$, let 
$f_H\colon \C^{d+1} \to \C$ be a linear form with kernel $H$. The product 
\begin{equation}
\label{eq:qa}
Q(\A)=\prod_{H\in \A} f_H,
\end{equation}
then, is a defining polynomial for the arrangement, 
unique up to a (non-zero) constant factor. 
Notice that $Q=Q(\A)$ is a homogeneous polynomial 
of degree equal to $\abs{\A}$, the cardinality of the set $\A$. 
The complement of the arrangement, 
\begin{equation}
\label{eq:comp}
M(\A)=\C^{d+1}\setminus \bigcup_{H\in\A}H,
\end{equation} 
is a connected, smooth complex quasi-projective variety.  Moreover,  
$M=M(\A)$ is a Stein manifold, and thus it has the homotopy type of a 
CW-complex of dimension at most $d+1$.  In fact, $M$ splits 
off the linear subspace  $\bigcap_{H\in \A} H$. 
Thus, we may safely assume that the arrangement 
$\A$ is {\em essential}, i.e., that this subspace is just $0$.

The group $\C^*$ acts freely on $\C^{d+1}\setminus \set{0}$ via 
$\zeta\cdot (z_0,\dots,z_{d})=(\zeta z_0,\dots, \zeta z_{d})$. 
The orbit space is the complex projective space of dimension $d$, 
while the orbit map, $\pi\colon \C^{d+1}\setminus \set{0} \to \CP^{d}$,  
$z \mapsto [z]$, 
is the Hopf fibration. The set  $\P(\A)=\set{\pi(H)\colon H\in \A}$ is an 
arrangement of codimension $1$ projective subspaces in $\CP^{d}$. 
Its complement, $U=U(\A)$, coincides with the quotient 
$\P(M)=M/\C^*$. 

The Hopf map restricts to a bundle map, $\pi\colon M\to U$, with fiber 
$\C^{*}$. Fixing a hyperplane $H\in \A$, we see that $\pi$ is 
also the restriction to $M$ of the bundle map 
$\C^{d+1}\setminus H\to \CP^{d} 
\setminus \pi(H) \cong \C^{d}$.  This latter bundle 
is trivial, and so we have a diffeomorphism  
$M \cong U\times \C^*$.

Fix now an order  $H_1,\dots ,H_n$ on the hyperplanes of $\A$, 
and denote the corresponding linear forms by $f_1,\dots , f_n$. 
We may then define a linear map $\iota\colon \C^{d+1}\to\C^n$  by 
$\iota(z)=(f_1(z), \dots ,f_n(z))$.   Since we assume $\A$ is essential, 
the map $\iota$ is injective. Its restriction to the complement yields an 
embedding $\iota\colon M\to (\C^*)^n$.  As shown in 
\cite{DeS-plms, Su-toul}, this embedding is is a classifying 
map for the universal abelian cover $M^{\ab}\to M$.

Clearly, the map $\iota\colon M\to(\C^*)^n$ 
is equivariant with respect to the diagonal action of $\C^*$ 
on both source and target.  Thus, $\iota$ descends to a map 
$\overline{\iota}\colon M/\C^*\to (\C^*)^n/\C^*$.  Since 
$\A$ is essential, this map  defines an embedding 
$\overline{\iota}\colon U \inj(\C^*)^{n-1}$, which is 
a classifying map for the universal abelian cover 
$U^{\ab}\to U$.

\subsection{The boundary manifold}
\label{subsec:bdry nbhd}
Let $V$ be the union of the hyperplanes in $\A$,  
and let $W=\P(V)$.  A regular neighborhood  of the algebraic 
hypersurface $W\subset \CP^{d}$ may be constructed as follows. 
Let $\phi\colon\CP^{d} \to \R$ be the smooth function 
defined by $\phi([z]) = \abs{Q(z)}^2 / \norm{z}^{2n}$, where 
$Q$ is a defining polynomial for the arrangement, 
and $n=\abs{\A}$. Then, for sufficiently small $\delta>0$, 
the space $\nu(W)=\phi^{-1} ([0,\delta])$ is a closed, regular 
neighborhood of $W$.   Alternatively, one may triangulate 
$\CP^{d}$ with $W$ as a subcomplex, and take $\nu(W)$ 
to be the closed star of $W$ in the second barycentric 
subdivision.  

As shown by Durfee \cite{Durfee}, these constructions yield 
isotopic neighborhoods, independent of the choices made.  
Plainly, $\nu(W)$ is a compact, orientable, smooth 
manifold with boundary, of dimension $2d$; 
moreover, $\nu(W)$ deform-retracts onto $W$. 
The  {\em exterior}\/ of the projectivized arrangement, 
denoted by $\Uc$, is the complement in $\CP^{d}$ 
of the  open regular neighborhood $\operatorname{int}(\nu(W))$. 
It is readily seen that $\Uc$ is a compact, connected, orientable, smooth 
$2d$-manifold with boundary, and that $U$ deform-retracts 
onto $\Uc$.

The \emph{boundary manifold}\/ of the arrangement 
$\A$ is the common boundary $\bdU= \partial \nu(W)$ 
of the exterior $\Uc$ and the regular neighborhood 
of $W$ defined above.  Clearly, $\bdU$ is a compact, 
orientable, smooth manifold of dimension $2d-1$.   
The inclusion map $\bdU \to \Uc$ is a  
$(d-1)$-equivalence, see Dimca \cite[Prop.~2.31]{Di92};  
in particular, $\pi_i(\bdU) \cong \pi_i(U)$ for $i<d-1$. 
Thus, $\bdU$ is connected if $d\ge 2$, and 
$\pi_1(\bdU)=\pi_1(U)$ if $d\ge 3$.   For more 
information on the boundary manifolds of arrangements, 
we refer to \cite{Hi01, CS06, CS08, Su-toul, KS}. 

\subsection{The Milnor fibration}
\label{sect:milnor fibration}

Once again, let $\A$ be a central arrangement of $n$ hyperplanes 
in $\C^{d+1}$.  
The polynomial map $Q=Q(\A)\colon \C^{d+1} \to \C$ restricts 
to a map $Q\colon M(\A) \to \C^{*}$, where 
$M=M(\A)$ is the complement of the arrangement. 
As shown by J.~Milnor \cite{Mi} in a more general context, this map 
is the projection map of a smooth, locally trivial bundle, 
known as the {\em Milnor fibration}\/ of the 
arrangement.  The typical fiber of this fibration, 
\begin{equation}
\label{eq:fa}
F(\A)=Q^{-1}(1)
\end{equation}
is called the {\em Milnor fiber}\/ of the arrangement.  
It is readily verified that $F=F(\A)$ is a smooth, connected, 
orientable manifold of dimension $2d$. 
Moreover, $F$ is a Stein domain of complex dimension $d$, 
and thus has the homotopy type of a finite CW-complex 
of dimension $d$. 

\begin{figure}
\[
\begin{tikzpicture}[baseline=(current bounding box.center),scale=0.66]  
\clip (0,0) circle (3.2);
\foreach \a in {0, 60,...,359}
      \draw(0, 0) -- (\a:8);
\draw[style=thick,color=red, scale=1,
domain=1:2.5,smooth,variable=\y] 
plot ({ 0.5*sqrt (\y^3-1)/sqrt(\y) },{\y});
\draw[style=thick,color=red, scale=1,
domain=1:2.5,smooth,variable=\y] 
plot ({ -0.5*sqrt (\y^3-1)/sqrt(\y) },{\y});
\draw[style=thick,color=red, scale=1.6,
domain=-1.25:-0.08,smooth,variable=\y] 
plot ({ 0.5*sqrt (1-\y^3)/sqrt(-\y) },{\y});
\draw[style=thick,color=red, scale=1.6,
domain=-1.25:-0.08,smooth,variable=\y] 
plot ({ -0.5*sqrt (1-\y^3)/sqrt(-\y) },{\y});
 \node at (0.7,1.5) [label=right:$\A$] {};
 \node at (1,-1) [label=right:$F(\A)$] {};
\end{tikzpicture}
\hspace*{0.7in}
\begin{tikzpicture}[baseline=(current bounding box.center),scale=0.66]   
        \draw (-1,0) to[bend left] (1,0);
        \draw (-1.2,.1) to[bend right] (1.2,.1);
        \draw[rotate=0] (0,0) ellipse (100pt and 50pt); 
        \node[circle,draw,inner sep=1pt,fill=black] at (2.2,0)  {};
        \node[circle,draw,inner sep=1pt,fill=black] at (-1.1,0.9)  {};
        \node[circle,draw,inner sep=1pt,fill=black] at (-1.1,-0.9)  {};
         \draw[->,>=stealth,shorten >=1pt,dashed] (1.9, 0.4) to[bend right] (-0.6,1);
         \node at (0.75,1.2) [label=right:$h$] {};
         \node at (1,-2) [label=right:$F(\A)$] {};
\end{tikzpicture}
\]
\caption{Milnor fiber and monodromy for a pencil of $3$ lines} 
\label{fig:mf pencil} 
\end{figure}
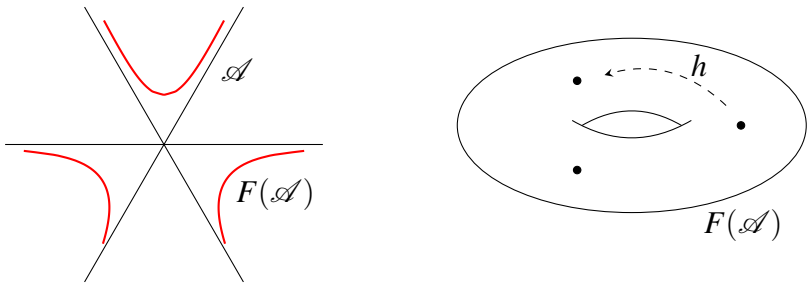

For each  $\theta\in [0,1]$, let us denote by $F_{\theta}$ the 
fiber over the point $e^{2\pi \ii \theta}\in \C^*$, so that $F_0=F_1=F$. 
For each point $z\in M$, the path $\gamma_{\theta}\colon [0,1] \to \C^*$, 
$t\mapsto e^{2\pi \ii t \theta}$ lifts to the path 
$\tilde\gamma_{\theta,z}\colon [0,1] \to M$ given by 
$t\mapsto e^{2\pi \ii t \theta/n} z$. Clearly,   
$Q(\tilde\gamma_{\theta,z}(1))= e^{2\pi \ii  \theta} Q(z)$.  
Therefore, if $z\in F_{0}$, then $\tilde\gamma_{\theta,z}(1)\in F_{\theta}$; 
moreover, $\tilde\gamma_{\theta,z}(0)=z$. 

By definition, the {\em monodromy}\/ of the Milnor fibration 
is the diffeomorphism $h\colon F_0\to F_1$ given by 
$h(z)=\tilde\gamma_{1,z}(1)$.  In view of the preceding  
discussion, this diffeomorphism can be written as 
$h\colon F\to F$, $z \mapsto e^{2\pi \ii/n} z$.  Clearly, 
$h$ has order $n$, and the complement $M$ is 
homotopy equivalent to the mapping torus of $h$.  

By homogeneity of the polynomial $Q$, 
we have that $Q(wz)=w^n Q(z)$, for every $z\in M$ 
and $w\in \C^{*}$. Thus, the restriction of $Q$ to a fiber of 
the Hopf bundle map $\pi\colon M\to U$ may be identified with 
the covering projection $q\colon \C^{*}\to \C^{*}$, $q(w)=w^n$.  
Now, if both $z$ and $wz$ belong to $F$, then $Q(z)=Q(wz)=1$, 
and so $w^n=1$. Thus, the restriction 
\begin{equation}
\label{eq:fucover}
\pi\colon F(\A) \to U(\A)
\end{equation}
is the orbit map of the free action of the geometric monodromy on 
$F(\A)$.   Hence, the Milnor fiber $F(\A)$ may be viewed as a regular, cyclic 
$n$-fold cover of the projectivized complement $U(\A)$, 
see for instance \cite{OR93, CS95, Su-toul}.  

\begin{example} 
\label{ex:mf boolean}
Let $\B_n$ be the Boolean arrangement in $\C^{n}$.   
Upon identifying the complement $M(\B_n)$ with the algebraic torus 
$(\C^*)^n$, we see that the map $Q(\B_n)\colon (\C^*)^n \to \C^*$, 
$z\mapsto z_1\cdots z_{n}$ is a morphism of complex algebraic groups. 
Hence, the Milnor fiber $F(\B_n)= \ker (Q(\B_n))$ 
is an algebraic subtorus, which is isomorphic to $(\C^*)^{n-1}$.  
The monodromy automorphism 
$h$ is isotopic to the identity, via the isotopy 
$h_t(z) = e^{2\pi \ii t/n} z$.  Thus, the Milnor fibration 
of the Boolean arrangement is trivial.
\end{example}  
 
As noted in \cite{Su-toul}, the map 
$\iota\colon M(\A)\inj M(\B_n)$ is compatible with the 
Milnor fibrations $Q(\A)\colon M(\A)\to \C^*$ 
and $Q(\B_n)\colon M(\B_n)\to \C^*$. It follows 
that the Milnor fiber $F(\A)$ may be obtained by intersecting the 
complement $M(\A)$, viewed as a subvariety of the algebraic 
torus $M(\B_n)=(\C^*)^n$ via the inclusion $\iota$, 
with $F(\B_n)\cong (\C^*)^{n-1}$, 
viewed as an algebraic subgroup of $(\C^*)^n$; that is, 
\begin{equation}
\label{eq:mfa}
F(\A) = M(\A) \cap F(\B_n).
\end{equation}

\subsection{The closed Milnor fiber and its boundary}
\label{subsec:bdF}

As before, let $\A$ be a (central) arrangement of hyperplanes 
in $\C^{d+1}$.  Intersecting the Milnor fiber $F(\A)$ with  
a ball in $\C^{d+1}$ of large enough radius, we obtain a 
compact, smooth, orientable $2d$-dimensional 
manifold with boundary,
\begin{equation}
\label{eq:fc}
\Fc(\A) = F(\A)\cap D^{2(d+1)}, 
\end{equation}
which we call the {\em closed Milnor fiber}\/ of the arrangement.  
The {\em boundary of the Milnor fiber}\/ of the arrangement $\A$ is   
the compact, smooth, orientable, $(2d-1)$-dimensional manifold
\begin{equation}
\label{eq:bdfma}
\bdF(\A) = F(\A)\cap S^{2d+1}.
\end{equation}

\begin{figure}
\centering
\begin{tikzpicture}[baseline=(current bounding box.center),scale=0.7]  
\clip (0,0) circle (3);
\draw[ style=thick, color=blue ] (-3,0) -- (3,0);
\draw[ style=thick, color=dkgreen] (0,-3) -- (0,3);
\draw[ ] (0,0) circle (2.4);
\draw[style=thick,color=orange, scale=1, domain=0.1:3,smooth,variable=\x] 
plot ({ \x },{1/\x});
\draw[style=thick,color=orange, scale=1, domain=-3:-0.1,smooth,variable=\x] 
plot ({ \x },{1/\x});
\end{tikzpicture}
\hspace*{0.3in}
\begin{tikzpicture}[baseline=(current bounding box.center),scale=0.72]  
\begin{scope}[even odd rule]
   \clip (0,0) circle(2.4) (1.8,0) circle (2.4);
   \fill[orange!40] (0,0) circle (2.4) (1.8,0) circle(2.4);
\end{scope}
\draw[dkgreen, very thick] (0,0)circle(2.4);
\draw[dkgreen!20!white, very thick] (-63:2.4) arc (-63:-73:2.4);
\draw[blue, very thick] (1.8,0)circle(2.4);
\draw[blue!20!white, very thick] (1.8,0)+(107:2.4) arc (107:117:2.4);
\draw[dkgreen, very thick] (63:2.4) arc (63:73:2.4); 
\end{tikzpicture}
\caption{Local Milnor fibration and closed Milnor fiber for $2$ lines} 
\label{fig:mf local} 
\end{figure}
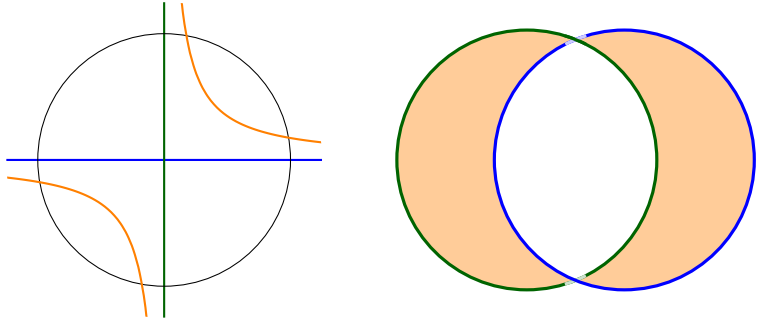

We will drop $\A$ from the notation when the arrangement is understood. 
As noted in \cite[Proposition 2.4]{Di92}, the pair $(\Fc, \bdF)$ 
is $(d-1)$-connected. In particular, if $d\ge 2$, the boundary 
of the Milnor fiber is connected, and the inclusion-induced 
homomorphism $\pi_1(\bdF)\to \pi_1(\Fc)$ is surjective. 
Furthermore, as shown in \cite[Lemma 7.5]{Su-toul}, the map 
$\pi \colon \C^{d+1} \setminus \set{0} \to \CP^{d}$ 
restricts to regular, cyclic $n$-fold covers, $\pi\colon \Fc \to \Uc$ 
and $\pi\colon \bdF \to \bdU$. 
In summary, we have a commuting ladder
\begin{equation}
\label{eq:cd}
\xymatrix{
\Z_n \ar@{=}[r] \ar[d] & \Z_n \ar@{=}[r] \ar[d] & \Z_n \ar[r] \ar[d] &
\C^* \ar@{=}[r] \ar[d] & \C^* \ar[d] \\
\bdF \ar^{\pi}[d] \ar[r]&  \Fc \ar^{\pi}[d] \ar^{\simeq}[r]&  F \ar^{\pi}[d] \ar[r]&  
M \ar[r] \ar^{\pi}[d] & \C^{d+1} \setminus\{0\} \ar^{\pi}[d] \\
\bdU  \ar[r]& \Uc  \ar^{\simeq}[r]& U  \ar@{=}[r]&  U \ar[r] & \CP^{d}
}
\end{equation}
where the horizontal arrows are inclusions, and the 
maps denoted by $\pi$ are principal bundles with 
fiber either $\Z_n$ or $\C^*$, as indicated.

\subsection{Classifying homomorphisms for cyclic covers}
\label{subsec:classify}

As before, let $\A$ be a central arrangement in $\C^{d+1}$, 
and set $n=\abs{\A}$.  
Fix a basepoint for the complement $M=M(\A)$.  
For each $H\in \A$, let $x_H$ denote the based 
homotopy class of a compatibly oriented meridian curve 
about the hyperplane $H$. A standard application of the 
van Kampen theorem shows that the fundamental group 
$\pi_1(M)$ is generated by these elements. To simplify  
notation, we will denote the image of $x_H$ in $H_1(M,\Z)$ 
by the same symbol.   Similarly, we will denote by 
$\overline{x}_H$ the image of $x_H$ in both 
$\pi_1(U)$ and its abelianization.  We then 
have that $H_1(M,\Z)$ is isomorphic to $\Z^n$, 
with basis $\set{x_H\colon H\in\A}$, and 
$H_1(U,\Z)$ is isomorphic to the quotient of 
$\Z^n$ by the cyclic subgroup generated by $\sum_{H\in\A} x_H$. 

Let $Q$ be a defining polynomial for $\A$, 
and let $Q\colon M\to \C^*$ be the Milnor fibration.  
By \cite[Prop.~4.6]{Su-toul}, the induced homomorphism 
$Q_{\sharp} \colon \pi_1(M)\to \pi_1(\C^*)=\Z$  
sends each generator $x_H$ to $1$. 
Recall that the Hopf fibration restricts 
to a regular, cyclic $n$-fold cover $\pi \colon F\to U$.  
As shown for instance in \cite{CS95, CDS03, Su-conm, Su-toul}, 
this cover is classified by the homomorphism 
$\delta\colon \pi_1(U) \surj \Z_n$ given by $\overline{x}_H\mapsto 1$.
If $d\ge 3$, we know that  $\pi_1(\bdU)=\pi_1(U)$, and so   
the $n$-fold cover $\pi\colon \bdF \to \bdU$ is classified by 
the same epimorphism $\delta$.  

In the critical case 
$d=2$, the $3$-dimensional manifold $\bdU$ is a graph manifold, 
with underlying graph $\Gamma$ the bipartite graph whose vertices 
correspond to the lines and the intersection points of the projectivized 
line arrangement in $\CP^2$, and with edges $(\ell,P)$ joining a line 
vertex $\ell$ to a point vertex $P$ if $P\in \ell$; see Figure \ref{fig:3lines} 
for an illustration.  Furthermore, each 
vertex manifold is the product of $S^1$ with a sphere $S^2$ 
with a number of open $2$-disks removed.  For more details 
on this construction we refer to  \cite{JY93, JY98, Hi01, CS08, KS}. 

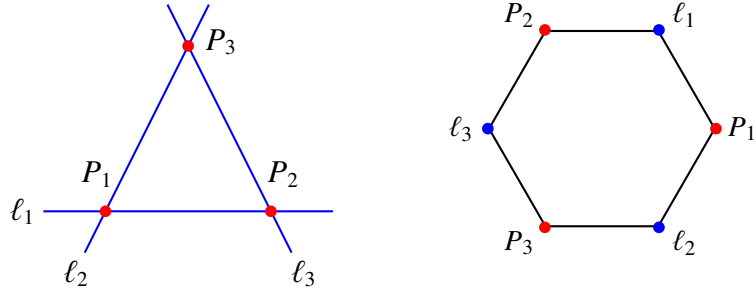
\begin{figure}
\centering
\begin{tikzpicture}[scale=0.55]
\draw[blue, style=thick] (-0.5,3) -- (2.5,-3);
\draw[blue, style=thick]  (0.5,3) -- (-2.5,-3);
\draw[blue, style=thick] (-3.5,-2) -- (3.5,-2);
\fill[red] (-2,-2) circle (4pt);  
\fill[red] (2,-2) circle (4pt);  
\fill[red] (0,2) circle (4pt);   
\node at (-4,-2) {$\ell_1$}; 
\node at (-2.7,-3.5) {$\ell_2$};
\node at (2.8,-3.5) {$\ell_3$};
\node at (-2.2,-1.1) {$P_1$}; 
\node at (2.3,-1.1) {$P_2$};
\node at (0.8,2.1) {$P_3$}; 
\node [draw, thick, minimum size=3cm, regular polygon, regular polygon sides=6]  (a) at (10,0) {};
\foreach \x in {1,3,5}
\fill[blue]  (a.corner \x) circle[radius=4pt];
\foreach \x in {2,4,6}
\fill[red] (a.corner \x) circle[radius=4pt];
\node[align=right, above] at (12,2.2) {\small $\ell_1$};
\node[align=right, below] at (12,-2.2) {\small $\ell_2$};
\node[align=left] at (6.6,0) {\small $\ell_3$};
\node[align=left, above] at (8,2.2) {\small $P_2$};
\node[align=left, below] at (8,-2.2) {\small $P_3$};
\node[align=right] at (13.4,0) {\small $P_1$};  
\end{tikzpicture}
\caption{An arrangement of lines and its associated graph}
\label{fig:3lines}
\end{figure}

The group $\pi_1(\bdU)$, then, has generators $\overline{x}_H$  
for $H\in \A$ and generators $y_c$ corresponding to the cycles in $\Gamma$. 
As shown in \cite[Prop.~7.6]{Su-toul}, the regular $\Z_n$-cover $\pi\colon \bdF \to \bdU$ 
is classified by the homomorphism
$\bar\delta\colon \pi_1(\bdU) \surj \Z_n$ 
given by $\overline{x}_H\mapsto 1$ and $y_c \mapsto 0$.

\section{Multinets and pencils}
\label{sect:multipen}

\subsection{Multinets}
\label{subsec:multinets}

For our purposes here, it will be enough to assume that 
the arrangement $\A$ lives in $\C^3$, in which case $\bar\A=\P(\A)$ is an 
arrangement of (projective) lines in $\CP^2$. This is clear when the 
rank of $\A$ is at most $2$, and may be achieved otherwise 
by taking a generic $3$-slice. This operation does not
change the poset $L_{\le 2}(\A)$, nor does it change 
the monodromy action on $H_1(F(\A),\C)$. 

For a rank-$3$ arrangement, the set $L_1(\A)$ is in $1$-to-$1$ correspondence 
with the lines of $\bar\A$, while $L_2(\A)$ is in $1$-to-$1$ correspondence 
with the intersection points of $\bar\A$.  Moreover, the poset structure of $L_{\le 2}(\A)$ 
mirrors the incidence structure of the point-line configuration $\bar\A$. 
We will say that a rank-$2$ flat $X$ has multiplicity $q$ if 
$\abs{\A_X}=q$, or, equivalently, if the point $\P(X)$ has 
exactly $q$ lines from $\bar\A$ passing through it.   The 
following notion, due to Falk and Yuzvinsky \cite{FY}, will play 
in an important role in the sequel. 

\begin{definition}[\cite{FY}]
\label{def:multinet}
A {\em multinet}\/ $\NN$ on an arrangement $\A$ consists 
of the following data:
\begin{romenum}
\item An integer $k\ge 3$, and a partition of $\A$ 
into $k$ subsets, say, $\A_1,\ldots,\A_k$.
\item An assignment of multiplicities on the hyperplanes, 
$m\colon \A\to \N$.
\item A subset $\XX\subseteq L_2(\A)$, called the base locus.
\end{romenum}
Moreover, the following conditions must be satisfied:
\begin{enumerate}
\item  \label{m1} 
There is an integer $d$ such that $\sum_{H\in\A_i} m_H=d$, 
for all $i\in [k]$.
\item  \label{m2} 
For any two hyperplanes $H$ and $K$ in different classes,  
$H\cap K\in \XX$.
\item  \label{m3} 
For each $X\in\XX$, the sum 
$n_X:=\sum_{H\in\A_i\colon H\supset X} m_H$ is independent of $i$.
\item  \label{m4} 
For each $i\in [k]$, the space 
$\big(\bigcup_{H\in \A_i} H\big) \setminus \XX$ is connected. 
\end{enumerate}
\end{definition}

\begin{figure}
\centering
\begin{tikzpicture}[scale=0.6]
\draw[style=thick,densely dashed,color=blue] (-0.5,3) -- (2.5,-3);
\draw[style=thick,densely dotted,color=red]  (0.5,3) -- (-2.5,-3);
\draw[style=thick,color=dkgreen] (-3,-2) -- (3,-2);
\draw[style=thick,densely dotted,color=red]  (3,-2.68) -- (-2,0.68);
\draw[style=thick,densely dashed,color=blue] (-3,-2.68) -- (2,0.68);
\draw[style=thick,color=dkgreen] (0,-3.1) -- (0,3.1);
\end{tikzpicture}
\hspace*{0.3in}
\begin{tikzpicture}[scale=0.6]
\draw[style=thick,color=dkgreen] (0,0) circle (3.1);
\node at (-2.4,0.4){2}; 
\node at (0,-2.6){2}; 
\node at (3.4,0.5){2}; 
\clip (0,0) circle (2.9);
\draw[style=thick,densely dashed,color=blue] (-1,-2.1) -- (-1,2.5);
\draw[style=thick,densely dotted,color=red] (0,-2.2) -- (0,2.5);
\draw[style=thick,densely dashed,color=blue] (1,-2.1) -- (1,2.5);
\draw[style=thick,densely dotted,color=red] (-2.5,-1) -- (2.5,-1);
\draw[style=thick,densely dashed,color=blue] (-2.5,0) -- (2.5,0);
\draw[style=thick,densely dotted,color=red] (-2.5,1) -- (2.5,1);
\draw[style=thick,color=dkgreen]  (-2,-2) -- (2,2);
\draw[style=thick,color=dkgreen](-2,2) -- (2,-2);
\end{tikzpicture}
\hspace*{0.3in}
\begin{tikzpicture}[scale=0.6]
\draw[style=thick,densely dotted,color=red] (0,0) circle (3.1);
\clip (0,0) circle (2.9);
\draw[style=thick,densely dotted,color=red] (0,-2.8) -- (0,2.8);  
\draw[style=thick,densely dotted,color=red] (-2.6,-1) -- (2.6,-1); 
\draw[style=thick,densely dotted,color=red] (-2.6,1) -- (2.6,1); 
\draw[style=thick,densely dashed,color=blue] (-0.5,-2.7) -- (-0.5,2.7); 
\draw[style=thick,densely dashed,color=blue] (1.5,-2.4) -- (1.5,2.4); 
\draw[style=thick,densely dashed,color=blue]  (-2.5,-2) -- (2.2,2.7); 
\draw[style=thick,densely dashed,color=blue](-2.2,1.7) -- (2.2,-2.7); 
\draw[style=thick,color=dkgreen] (-1.5,-2.4) -- (-1.5,2.4);  
\draw[style=thick,color=dkgreen] (0.5,-2.7) -- (0.5,2.7);  
\draw[style=thick,color=dkgreen]  (-1.7,-2.2) -- (2.2,1.7);  
\draw[style=thick,color=dkgreen](-2,2.5) -- (2.2,-1.7);  
\end{tikzpicture}
\caption{A $(3,2)$-net; a $(3,4)$-multinet; and a reduced $(3,4)$-multi\-net 
which is not a $3$-net}
\label{fig:multinets}%
\end{figure}
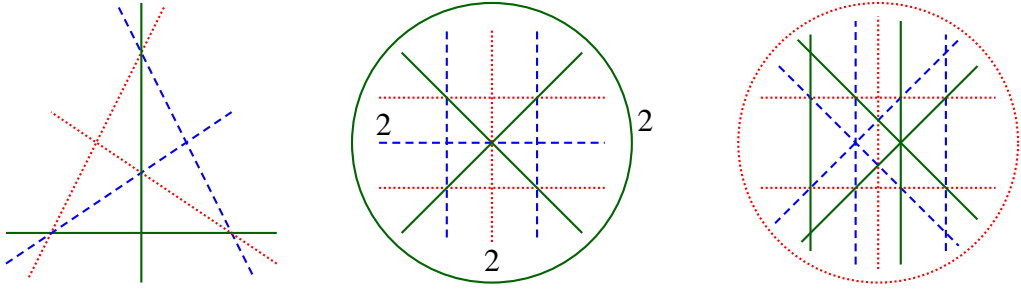

We say that a multinet  as above has $k$ classes and weight $d$, 
and refer to it as a $(k,d)$-multinet, or simply as a $k$-multinet.    
Without essential loss of generality, we may assume that $\gcd\{m_H\}_{H\in \A}=1$.  
If all the multiplicities are equal to $1$, the multinet is said to be {\em reduced}. 
If, furthermore, every flat in $\XX$  is contained in precisely one 
hyperplane from each class, the multinet is called 
a {\em $(k,d)$-net}. 

The various possibilities are illustrated in Figure \ref{fig:multinets}. 
The first picture shows a $(3,2)$-net on a planar slice 
of the reflection arrangement of type ${\rm A}_3$.  
The second picture shows a non-reduced $(3,4)$-multinet 
on a planar slice of the reflection arrangement of type ${\rm B}_3$.  
Finally, the third picture shows a simplicial arrangement of $12$ 
lines in $\CP^2$ supporting a reduced $(3,4)$-multinet which is not a $3$-net. 

Work of Yuzvinsky \cite{Yu04, Yu09} and Pereira--Yuzvinsky \cite{PY} 
shows that, if $\NN$ is a $k$-multinet on an arrangement $\A$, 
with base locus of size greater than $1$, then $k=3$ 
or $4$; moreover, if the multinet $\NN$ is not reduced, then $k=3$. 
Although several infinite families of multinets with $k=3$ are known, 
only one multinet with $k=4$ is known to exist: the $(4,3)$-net 
on the Hessian arrangement.  For more examples and further 
discussion, we refer to \cite{BYu, FY, TY, Yo}.

As noted in \cite[Lemma 2.1]{PS-betamilnor}, 
if $\A$ has no $2$-flats of multiplicity $kr$, 
for any $r>1$, then every reduced $k$-multinet on $\A$ is a $k$-net. 
The next lemma provides an alternative definition of nets. 

\begin{lemma}[\cite{PS-betamilnor}]
\label{lem:lsq}
A $k$-net on an arrangement $\A$ is a partition with non-empty blocks,
$\A =\coprod_{\alpha \in [k]} \A_{\alpha}$, with the property that, 
for every $H\in \A_{\alpha}$ and $K\in \A_{\beta}$ with $\alpha \ne \beta$ 
we have that $\abs{H\cap K \cap \A_{\gamma}}=1$, for every $\gamma\in [k]$.  
\end{lemma}

In particular, a $3$-net on $\A$ is a partition into 
$3$ non-empty subsets with the 
property that, for each pair of hyperplanes $H,K\in \A$ in 
different classes, we have $H\cap K=H\cap K\cap L$, for some 
hyperplane $L$ in the third class.
Nets of type $(3,d)$ are intimately related to Latin squares 
of size $d$, i.e., $d\times d$ matrices 
with each row and column a permutation of the set $[d]$.   
Indeed, if  $\A_1,\A_2,\A_3$ are the parts of such a $3$-net, then 
the multi-colored $2$-flats define a Latin square: 
if we label the hyperplanes of $\A_{\alpha}$ as 
$H^{\alpha}_{1},\dots, H^{\alpha}_{d}$, 
then the $(p,q)$-entry of this matrix is the integer $r$ 
given by the requirement that 
$H^1_p \cap H^2_q \cap H^3_{r}  \in L_2(\A)$.
A similar procedure shows that a $k$-net is 
encoded by a $(k-2)$-tuple of orthogonal Latin squares. 

\subsection{Pencils}
\label{subsec:pens}
Let $\A$ be a (central) arrangement in $\C^3$, with defining 
polynomial $Q(\A)=\prod_{H\in \A} f_H$. Suppose we 
have a $(k,d)$-multinet $\NN$ on $\A$, 
with parts $\A_\alpha$ and multiplicity vector $m$.  
Write $Q_{\alpha}=\prod_{H\in\A_{\alpha}}f_H^{m_H}$, 
and define a rational map 
$f\colon \C^3\to\CP^1$ by $f(x)=(Q_1(x) \colon Q_2(x))$. 
There is then a set 
$D=\set{(a_1\colon b_1), \dots , (a_k\colon b_k)}$ 
of $k$ distinct points in $\CP^1$
such that each of the degree $d$ polynomials $Q_1,\dots, Q_k$ 
can be written as $Q_{\alpha}=a_{\alpha} Q_2-b_{\alpha}Q_1$, and, 
furthermore, the image of $f\colon M(\A)\to\CP^1$ misses $D$.   
The {\em pencil}\/ associated to the multinet $\NN$, then, 
is the restriction of $f$ to the complement of the arrangement,  
\begin{equation}
\label{eq:pen}
f=f_{\NN}\colon M(\A) \to \CP^1\setminus D.
\end{equation}

The map $f$ can also be viewed as an `orbifold fibration,' or, in the 
terminology of Arapura \cite{Ar}, an `admissible map.' 
To compute the homomorphism induced in first homology by this 
map, let $\gamma_1,\dots ,\gamma_k$ be compatibly oriented  
simple closed curves on $S=\CP^1\setminus D$, going around 
the points of $D$, so that $H_1(S,\Z)$ is generated by the homology 
classes $c_{\alpha}=[\gamma_{\alpha}]$, subject to the single relation 
$\sum_{\alpha=1}^k c_{\alpha}=0$.  The following lemma was proved 
in \cite{PS-betamilnor} using an approach based on de Rham cohomology.  
We give here another proof.  

\begin{lemma}[\cite{PS-betamilnor}]
\label{lem:pen h1}
Let $f\colon M(\A) \to S$ be the pencil 
associated to a multinet $\NN$ on an arrangement $\A$. 
The induced homomorphism $f_{*} \colon H_1(M(\A),\Z) \to H_1(S,\Z)$ 
is then given by 
\begin{equation*}
\label{eq:multi hom}
f_*(x_H) = m_H c_{\alpha}, \quad\text{for  $H\in \A_{\alpha}$}.
\end{equation*}
\end{lemma}

\begin{proof}
Each polynomial $Q_{\alpha}$ defines a map 
$Q_{\alpha}\colon M(\A_{\alpha})\to \C^*$. 
If we let $\theta_{\alpha}\colon S \to \C^*$ be the 
map given by $\theta_{\alpha}(z_1\colon z_2) = a_{\alpha}z_2-b_{\alpha} z_1$, 
and let $\iota_{\alpha}\colon M(\A)\to M(\A_{\alpha})$ be the inclusion map, 
we obtain a commuting diagram,
\begin{equation}
\label{eq:cd-bis}
\xymatrix{M(\A) \ar^(.54){f}[r] \ar^{\iota_{\alpha}}[d] 
& S \ar^{\theta_{\alpha}}[d] \\
M(\A_{\alpha}) \ar^(.56){Q_{\alpha}}[r] & \C^*}
\end{equation}

Apply now the $H_1(-,\Z)$  functor to this diagram, 
and identify $H_1(\C^*,\Z)=\Z$.   Clearly,  if $H\in \A_{\beta}$, then 
$(\iota_{\alpha})_*$ takes $x_H$ to $\delta_{\alpha\beta} a_H$, 
whereas $(\theta_{\alpha})_*$ takes $c_{\beta}$ to $\delta_{\alpha\beta}$, 
where $\delta_{\alpha\beta}$ is the Kronecker delta. 
On the other hand, $(Q_{\alpha})_*$ is given by $x_H\mapsto m_H$, 
see \cite[Prop.~4.6]{Su-toul}.  This completes the proof.
\end{proof}

\section{Cohomology jump loci}
\label{subsec:cjl}

\subsection{Resonance varieties of a graded algebra}
\label{subsec:res var}

Let $A$ be a graded, graded-commutative algebra over a 
field $\k$.  We will assume that each 
graded piece $A^i$ is free and finitely generated over $\k$, 
and $A^0=\k$. We will also assume that $a^2=0$, for all $a\in A^1$, 
a condition which is automatically satisfied if $\ch(\k)\ne 2$, 
by graded-commutativity of multiplication in $A$.
For each element $a \in A^1$, we turn the algebra $A$ into a 
cochain complex, 
\begin{equation}
\label{eq:aomoto}
\xymatrix{(A , \delta_a)\colon  \ 
A^0\ar^(.66){ \delta_{a}}[r] & A^1\ar^{ \delta_{a}}[r] & A^2  \ar[r]& \cdots},
\end{equation}
with differentials the maps  $ \delta_{a}(b)=ab$. 
The (degree $i$, depth $s$) {\em resonance varieties}\/ of $A$ 
are then defined as the jump loci for the cohomology of this `Aomoto' complex, 
\begin{equation}
\label{eq:resa}
\RR^i_s(A)= \{a \in A^1 \mid \rank_{\k} H^i(A , \delta_{a}) \ge  s\}.
\end{equation}

These sets are Zariski-closed, homogeneous subsets of the affine 
space $A^1$.  Furthermore, these varieties respect field extensions:  
if $\k\subseteq \mathbb{K}$, then
$\RR^i_s(A)=\RR^i_s(A \otimes \mathbb{K}) \cap A^1$.
As shown in \cite{PS-plms, SW}, the resonance varieties 
obey the following `product formulas':
\begin{align}
\label{eq:resprod}
\RR^1_s(A \otimes B)&=\RR^1_s(A)\times \{0\} \cup \{0\}\times \RR^1_s(B),\\
\RR^i_1(A \otimes B)&=\bigcup\limits_{j+k=i} \RR^j_1(A)\times  \RR^k_1(B). \notag
\end{align}

For our purposes here, we will mainly consider the degree 
one resonance varieties, $\RR_s(A)=\RR^1_s(A)$. Clearly, these varieties 
depend only on the degree $2$ truncation of $A$.  More explicitly, 
$\RR_s(A)$ consists of $0$, together with all elements 
$a \in A^1$ for which there exist $b_1,\dots ,b_s \in A^1$ 
such that the span of $\{a,b_1,\dots,b_s\}$ has dimension $s+1$ and 
$ab_1=\cdots =ab_s=0$ in $A^2$.

The degree $1$ resonance varieties enjoy the following naturality 
property:  if $\varphi\colon A\to B$ is a morphism of commutative 
graded algebras, and  $\varphi$ is injective in degree $1$, then 
the $\k$-linear map $\varphi^1\colon A^1\to B^1$ embeds $\RR_s(A)$ 
into $\RR_s(B)$, for each $s\ge 1$.

Finally, suppose $X$ is a connected, finite-type CW-complex. 
We define then the resonance varieties of $X$ to be the sets 
$\RR^i_s(X,\k):=\RR^i_s(H^*(X,\k))$, viewed as homogeneous 
subsets of the affine space $H^1(X,\k)$.

\subsection{The resonance varieties of the Orlik--Solomon algebra}
\label{subsec:os alg}

The cohomology ring of a hyperplane arrangement 
complement was computed by E.~Brieskorn in the early 1970s, 
building on pioneering work of V.I.~Arnol'd on the cohomology 
ring of the pure braid group.   In \cite{OS}, Orlik and 
Solomon gave a simple description of this ring, solely 
in terms of the intersection lattice of the arrangement.  

Once again, let $\A$ be a central arrangement, with complement 
$M=M(\A)$.  Fix a linear order on $\A$, and let $E$ be the exterior 
algebra over a field $\k$ with generators $\set{e_H \mid H\in \A}$ 
in degree $1$.  Next, define a differential $\partial \colon E\to E$ 
of degree $-1$, starting from $\partial(1)=0$ 
and $\partial(e_H)=1$, and extending $\partial$ to 
a linear operator on $E$, using the graded Leibniz rule.  
Finally, let $I$ be the ideal of $E$ generated by all elements 
of the form $\partial \big(\prod_{H\in \B} e_H\big)$, 
where $\B\subset \A$ and  
$\codim \bigcap_{H\in \B} H < \abs{\B}$.  Then $H^*(M,\k)$ is 
isomorphic, as a graded $\k$-algebra, to the quotient ring $A=E/I$.  

Under this isomorphism, the basis $\{e_H\}$ of
$A^1$ is dual to the basis of $H_1(M,\k)=H_1(M,\Z)\otimes \k$
given by the meridians $\{x_H\}$ around the hyperplanes, 
oriented compatibly with the complex orientations on $\C^{d+1}$ 
and the hyperplanes. 
Since $A$ is a quotient of an exterior algebra, we have that 
$a^2=0$ for all $a\in A^1$. Thus, we may define the resonance varieties 
$\RR^i_s(\A,\k)$ of our arrangement $\A$ (over the field $\k$) as the 
corresponding resonance varieties 
of the Orlik--Solomon algebra $H^*(M(\A),\k)$. 

As usual, let $U=U(\A)$ be the projectivized complement.  
The diffeomorphism $M\cong U\times \C^*$, together with the 
K\"{u}nneth formula and the product formulas for resonance 
from \eqref{eq:resprod} yields identifications
\begin{align}
\label{eq:rqma}
\RR^1_s(\A,\k)&\cong \RR^1_s(U(\A),\k),\\
\RR^i_1(\A,\k) &\cong \RR^i_1(U(\A),\k)\cup 
\RR^{i-1}_1(U(\A),\k).  \notag
\end{align}  

If $\B\subset \A$ is a proper sub-arrangement, the inclusion 
$M(\A) \inj M(\B)$ induces a morphism 
$H^*(M(\B),\k) \to H^*(M(\B),\k)$ which is injective 
in degree $1$.  We thus obtain an embedding 
$\RR^1_s(\B,\k) \inj \RR^1_s(\A,\k)$.  The irreducible 
components of $\RR^1_s(\A,\k)$ that  lie in the image 
of such an embedding are called {\em non-essential}; 
the remaining components are called {\em essential}. 

The description of the Orlik--Solomon algebra given above 
makes it clear that the resonance 
varieties $\RR^i_s(\A,\k)$ depend only on the 
intersection lattice, $L(\A)$, and on the characteristic 
of the field $\k$.  

The complex resonance varieties $\RR^i_1(\A,\C)$ 
were first defined and studied by Falk in \cite{Fa97}.
Soon after, Cohen--Suciu \cite{CS99}, Libgober \cite{Li01}, and 
Libgober--Yuzvinsky \cite{LY00} showed that the varieties 
$\RR_s(\A)=\RR^1_s(\A,\C)$ consist of linear subspaces 
of the vector space $\C^{\A}$, intersecting transversely at $0$.  
Moreover, all such subspaces  
have dimension at least two, and the cup-product map 
vanishes identically on each one of them. 
Finally, $\RR_s(\A)$ is the union of all components 
of $\RR_1(\A)$ of dimension greater than $s$.

The resonance varieties $\RR^1_s(\A,\k)$ for $\k$ a field of positive 
characteristic were first defined and studied by Matei and Suciu in \cite{MS00}.  
The nature of these varieties is much less predictable; for instance, their 
irreducible components need not be linear, and, even when they are linear, 
they may intersect non-transversely.  We refer to 
\cite{S01, Fa07, DeS-plms, PS-betamilnor} for more on this subject. 

\subsection{Multinets and complex resonance}
\label{subsec:multi res}
Work of Falk and Yuzvinsky \cite{FY} greatly clarified 
the nature of the (degree $1$) resonance varieties of arrangements. 
Let us briefly  review their construction.

Recall that every $k$-multinet $\NN$ on an arrangement $\A$ 
with parts $\A_1,\dots, \A_k$ and multiplicities $m_H$ for each $H\in \A$ 
gives rise to an orbifold fibration (or, for short, a pencil)  
$f\colon M(\A) \to S$, where $S= \CP^1\setminus 
\{\text{$k$ points}\}$.   In view of Lemma \ref{lem:pen h1}, the induced 
map in cohomology, $f^{*} \colon H^*(S,\Z) \to H^*(M,\Z)$, is given 
in degree $1$ by $f^*(c_{\alpha}^*) = u_{\alpha}$, where
$u_{\alpha}=\sum_{H\in \A_{\alpha}} m_H e_H$.  Consequently,  
the homomorphism $f^{*} \colon H^1(S,\C) \to H^1(M,\C)$ is injective, and thus 
sends $\RR_1(S,\C)$ to $\RR_1(M,\C)$.

Let us identify $\RR^1(S,\C)$ with $H^1(S,\C)=\C^{k-1}$, and view 
$P_{\NN}=f^*(H^1(S,\C))$ as lying inside $\RR_1(\A)$.  It follows from 
the preceding discussion that  $P_{\NN}$ is the $(k-1)$-dimensional linear 
subspace spanned by the vectors $u_2-u_1,\dots , u_k-u_1$.  In fact, 
as shown in \cite[Thms.~2.4--2.5]{FY}, this subspace is 
an essential component of $\RR_1(\A)$. 

More generally, suppose there is a sub-arrangement 
$\B\subseteq \A$ supporting a multinet $\NN$.   In this case, 
the inclusion $M(\A) \inj M(\B)$ induces a 
monomorphism $H^1(M(\B),\C) \inj H^1(M(\A),\C)$, 
which restricts to an embedding $\RR_1(\B) \inj \RR_1(\A)$.  
The linear space $P_{\NN}$, then, 
lies inside $\RR_1(\B)$, and thus, inside $\RR_1(\A)$.
Conversely, as shown in \cite[Thm.~2.5]{FY} 
all (positive-dimensional) irreducible components 
of $\RR_1(\A)$ arise in this fashion. Consequently, 
\begin{equation}
\label{eq:rsa}
\RR_s(\A) = \bigcup_{\B \subseteq \A} 
\bigcup_{\stackrel{\text{\tiny{$\NN$ a multinet on $\B$}}}
{\text{\tiny{with at least $s+2$ parts}}}} P_{\NN}.
\end{equation}

\subsection{Characteristic varieties and finite abelian covers}
\label{subsec:cv}

We switch now to a different type of jump loci, involving this 
time homology with twisted coefficients. 
Let $X$ be a connected, finite-type CW-complex, let 
$\pi=\pi_1(X,x_0)$, and let $\Hom(\pi,\C^*)$ be the affine 
algebraic group of $\C$-valued, multiplicative characters on $\pi$, 
which we will identify with $H^1(X,\C^*)$. 
The (degree $i$, depth $s$) {\em characteristic varieties}\/ of 
$X$ are the jump loci for homology with coefficients in rank-$1$ 
local systems on $X$:
\begin{equation}
\label{eq:cvs}
\VV^i_s(X)=\{\xi\in\Hom(\pi,\C^*)  \mid  
\dim_{\C} H_i(X,\C_\xi)\ge s\}.
\end{equation}
By construction, these loci are Zariski-closed subsets of 
the character group. 

To a large degree, the characteristic varieties 
control the Betti numbers of regular, finite abelian covers $Y\to X$.  
For instance, suppose that the deck-trans\-formation group is 
cyclic of order $n$, and fix an inclusion $\iota \colon \Z_n \inj \C^*$, 
by sending $1 \mapsto e^{2\pi \ii/n}$. With this choice, the 
epimorphism $\nu \colon \pi\surj \Z_n$ defining the $n$-fold 
cyclic cover $Y \to X$ yields a torsion character, $\rho=\iota\circ \nu\colon \pi \to \C^*$. 
A standard argument using Maschke's theorem yields an isomorphism 
of $\C[\Z_n]$-modules,
\begin{equation}
\label{eq:equiv}
H_i(Y,\C) \cong H_i(X,\C) \oplus 
\bigoplus_{1<r\mid n} (\C[t]/\Phi_r(t))^{\depth(\rho^{n/r})}, 
\end{equation}
where $\Phi_r(t)$ is the $r$-th 
cyclotomic polynomial, and the depth of a character $\xi\colon \pi\to \C^*$, 
defined as the dimension of $H_i(X, \C_{\xi})$, is given by 
$\depth(\xi)=\max\{s \mid \xi\in \VV^i_s(X)\}$. 

The exponents in formula \eqref{eq:equiv} 
arising from prime-power divisors can be estimated in terms of the 
corresponding Aomoto--Betti numbers. More precisely, suppose 
$n$ is divisible by $r=p^s$, for some prime $p$.  Composing the 
canonical projection $\Z_n \surj \Z_p$ with $\nu$ defines a 
cohomology class $\bar\nu\in H^1(X,\F_p)$.  
Assuming that $H_*(X,\Z)$ is 
torsion-free, it was shown in \cite[Thm.~11.3]{PS-tams} that 
\begin{equation}
\label{eq:modular ineq}
\dim_{\C} H_i(X, \C_{\rho^{n/r}}) \le 
\dim_{\F_p} H^i( H^{\hdot}(X, \F_p), \delta_{\bar\nu}).
\end{equation}

\subsection{Characteristic varieties of arrangements}
\label{subsec:char var}
Let us consider again a  hyperplane arrangement $\A$, 
with complement $M=M(\A)$. 
The varieties $\VV_s(\A):=\VV^1_s(M(\A))$ are closed algebraic 
subsets of the character torus $\Hom(\pi_1(M),\C^*)=(\C^*)^n$, 
where $n=\abs{\A}$.  
Since $M$ is diffeomorphic to $U\times \C^*$, where $U=U(\A)$, 
the character torus $H^1(M,\C^*)$ splits as 
$H^1(U,\C^*)\times \C^*$.  Under this splitting, 
the characteristic varieties $\VV_s(\A)$ get 
identified with the varieties $\VV^1_s(U)$ lying 
in the first factor. 

Since $M$ is a smooth, quasi-projective variety, 
a general result of Arapura \cite{Ar}, as refined by 
Artal Bartolo--Cogolludo--Matei \cite{ACM} 
and Budur--Wang \cite{BW}, 
insures that $\VV_s(\A)$ is a finite union of translated 
subtori.  Moreover, as shown by Cohen--Suciu \cite{CS99} 
and Libgober--Yuzvinsky \cite{LY}, and, 
in a broader context, by Dimca--Papadima--Suciu \cite{DPS-duke} 
and Dimca--Papadima \cite{DP-ccm}, 
the tangent cone at the origin to $\VV_s(\A)$ coincides with the 
resonance variety $\RR_s(\A)$, for all $s\ge 1$. 

More explicitly, let 
$\exp\colon H^1(M,\C)\to H^1(M,\C^*)$ be the 
coefficient homomorphism induced by the exponential map 
$\C\to \C^*$. 
Then, if $P\subset H^1(M,\C)$ is one of the linear 
subspaces comprising $\RR_s(\A)$, its image under 
the exponential map, $\exp(P)\subset H^1(M,\C^*)$, 
is one of the subtori comprising $\VV_s(\A)$. Furthermore, 
the correspondence $P\leadsto \exp(P)$ yields 
a bijection between the components of $\RR_s(\A)$ 
and the components of $\VV_s(\A)$ passing through 
the origin $\bo$.  

Now recall that each positive-dimensional component of $\RR_1(\A)$ 
is obtained by pullback along a pencil $f\colon M\to S$, where $S=\CP^{1}\setminus 
\{\text{$k$ points}\}$ and $k\ge 3$. Thus, each positive-dimensional 
component of $\VV_1(\A)$ containing the origin is of the form 
$\exp(P)=f^*(H^1(S,\C^*))$, for some pencil $f$. 
An easy computation shows that 
$\VV^1_s(S)=H^1(S,\C^*)=(\C^*)^{k-1}$ for all 
$s \le k-2$.  Hence, the subtorus $f^*(H^1(S,\C^*))$ is a 
positive-dimensional component of $\VV_1(\A)$ that contains the origin 
and lies inside $\VV_{k-2}(\A)$. 

As shown in \cite{S02}, the (depth $1$) characteristic variety of an arrangement 
may have irreducible components not passing through the origin.  
A general combinatorial machine for producing such translated 
subtori has been recently given in \cite{DeS-plms}. 
Namely, suppose $\A$ admits a {\em pointed multinet}, 
i.e., a multinet $\NN$ and a hyperplane $H\in \A$ 
for which $m_H>1$, and $m_H \mid n_X$ 
for each flat $X$ in the base locus such that $X\subset H$.  
Letting $\A'=\A\setminus \{H\}$ 
be the deletion of $\A$ with respect to $H$, it turns out 
that $\VV_1(\A')$ has a component which is a $1$-dimensional 
subtorus of $H^1(M(\A'),\C^*)$, translated by a character of 
order $m_H$.  

For instance, if $\A$ is the reflection arrangement of type ${\rm B}_3$ 
and $\NN$ is the $(4,3)$-multinet depicted in the middle of Figure \ref{fig:multinets}, 
then choosing $H$ to be one of the hyperplanes with multiplicity $m_H=2$ 
leads to a translated torus in the first characteristic variety of the 
deleted ${\rm B}_3$ arrangement, $\A'=\A\setminus \{H\}$.  
Whether all positive-dimensional translated 
subtori in the (degree $1$, depth $1$) characteristic varieties 
of arrangements occur in this fashion is an open problem.

\section{The algebraic monodromy of the Milnor fiber}
\label{sect:homology}

\subsection{The homology of the Milnor fiber}
\label{subsec:milnor}

Using the interpretation of the Milnor fiber of a hyperplane 
arrangement as a finite cyclic cover of the projectivized complement, 
we may compute the homology groups of the Milnor fiber and the 
characteristic polynomial of the algebraic monodromy in terms 
of the characteristic varieties of the arrangement. 

To see how that works, let $\A$ be an arrangement of $n$ hyperplanes 
in $\C^{d+1}$.  Without loss of generality, we may assume  $d=2$.  
Let $M$ be the complement of the arrangement, and let $U$ be its projectivization.  
Recall that the Milnor fiber $F=F(\A)$ may be viewed as the regular, $\Z_n$-cover 
of $U$, classified by the homomorphism $\pi_1(U)\surj \Z_n$ taking each 
meridian loop $x_H$ to $1$.  

For each divisor $r$ of $n$, let $\rho_r\colon \pi_1(U)\to \C^*$ 
be the character defined by $\rho_r(x_H)= e^{2\pi \ii /r}$. 
It follows from formula \eqref{eq:equiv} that 
\begin{equation}
\label{eq:h1milnor}
H_1(F(\A),\C) = H_1(U,\C) \oplus \bigoplus_{1<r\mid n} 
(\C[t]/\Phi_r(t))^{e_r(\A)},
\end{equation}
as modules over $\C [\Z_n]$, where the integers $e_r(\A):=\depth (\rho_r)$ 
depend on the position of the diagonal characters $\rho_r\in (\C^*)^{n-1}$ 
with respect to the characteristic varieties $\VV_s(U)$.  
Note that only the essential components of these varieties 
may contribute to the sum.  Indeed, if a component 
lies on a proper coordinate subtorus $C$, 
then the diagonal subtorus, $D=\{(t,\dots ,t) \mid t\in \C^*\}$, 
intersects $C$ only at the origin. In particular, components 
arising from multinets supported 
on proper sub-arrangements of $\A$, do not produce 
jumps in the first Betti number of $F(\A)$.  

Let $h_*\colon H_1(F,\C)\to H_1(F,\C)$ be the 
degree $1$ algebraic monodromy of the Milnor fibration, and 
let  $\Delta_{\A}(t)=\det (t \cdot \id - h_*)$ be its characteristic polynomial.  
Formula \eqref{eq:h1milnor} may be interpreted as saying that 
\begin{equation}
\label{eq:delat}
\Delta_{\A}(t) = (t-1)^{n-1}  \cdot \prod_{1<r\mid n} \Phi_r(t)^{e_r(\A)}.
\end{equation}

Consequently, if $\varphi(r)$ denotes the Euler totient function, then 
\begin{equation}
\label{eq:betti1 mf}
b_1(F(\A))= n-1 + \sum_{1< r | n} \varphi(r)e_r(\A). 
\end{equation}

In the above expressions, not all the divisors $r$ of $n$ appear. 
For instance, as shown by Libgober \cite[Prop.~2.1]{Li02} and 
M\u{a}cinic--Papadima \cite[Thm.~3.13]{MP}, the 
following holds: if there is no flat $X\in L_2(\A)$ of 
multiplicity $q\ge 3$ such that $r\mid q$, then $e_r(\A)$ vanishes. 
In particular, if the lines of $\bar\A$ intersect  
only in points of multiplicity $2$ and $3$, then only $e_3(\A)$ 
may be non-zero, whereas if points of multiplicity $4$ 
occur, then $e_2(\A)$ and $e_4(\A)$ may also be non-zero.  
For more combinatorial conditions that lead to the vanishing 
of the exponents $e_r(\A)$ we refer to \cite{CS95, CDO03, Bt14, BY16}. 

In \cite[Thm.~13]{CDO03}, Cohen, Dimca, and Orlik 
give combinatorial upper bounds on the exponents 
of the cyclotomic polynomials appearing in \eqref{eq:h1milnor}.  
The next result provides lower bounds for those exponents, 
in the presence of reduced multinets on the arrangement.

\begin{theorem}[\cite{PS-betamilnor}]
\label{thm:be3}
Suppose that an arrangement $\A$ admits a reduced $k$-multinet, and 
let $f\colon M(\A)\to S$ denote the associated pencil.  Then: 
\begin{enumerate}
\item \label{bb1}
The character $\rho_k$ belongs to $f^*(H^1(S,\C^*))$, and  
$e_k(\A)\ge k-2$. 
\item \label{bb2}
If $k=p^s$, then $\rho_{p^r}\in f^*(H^1(S,\C^*))$
and $e_{p^r}(\A)\ge k-2$, for all $1\le r\le s$.
\end{enumerate}
\end{theorem}

\subsection{Aomoto--Betti numbers}
\label{subsec:aomoto betti}

Consider now a field $\k$ of characteristic $p$, and let $A=H^{*}(M,\k)$ 
be the Orlik--Solomon algebra over $\k$.   Recall that the $\k$-vector 
space $A^1=\k^{\A}$ comes endowed with a preferred basis, 
$\{e_H\}_{H\in \A}$; let $\sigma=\sum_{H\in  \A} e_H$ 
be the ``diagonal" element. Following \cite{PS-betamilnor}, we define 
the {\em Aomoto--Betti number}\/ of $\A$ (over $\k$) as 
\begin{equation}
\label{eq:betap}
\beta_\k(\A) = \max\{s \mid \sigma\in \RR_s(\A,\k)\}.
\end{equation} 

Clearly, this integer depends only on the prime $p=\ch (\k)$, and so we will 
write it simply as $\beta_{p}(\A)$.  The following result provides useful 
information about these combinatorial invariants of arrangements. 

\begin{prop}[\cite{PS-betamilnor}]
\label{prop:betap}
Let $\A$ be an arrangement, and $p$ a prime. 
\begin{enumerate}
\item If $p\nmid \abs{X}$, for any $X\in L_2(\A)$ with $\abs{X}>2$, 
then $\beta_p(\A)=0$. 
\item
If $\A$ supports a $k$-net, then $\beta_p (\A)=0$ if $p\not\mid k$,
and $\beta_p (\A) \ge k-2$, otherwise.
\end{enumerate}
\end{prop}

To a large extent, the $\beta_p$ invariants control the (degree $1$) 
algebraic monodromy of the Milnor fibration.  More precisely, 
the ``modular upper bound" \eqref{eq:modular ineq}  yields the 
following inequalities on the prime-power exponents, 
\begin{equation}
\label{eq:bound bis}
e_{p^s} (\A) \le \beta_p(\A),  
\end{equation}
for all primes $p$ and integers $s\ge 1$.  In particular, if $\beta_p(\A)=0$, 
then $e_{p^s} (\A) =0$, for all $s\ge 1$.

\subsection{Nets, multiplicities, and the Milnor fibration}
\label{subsec:3nets}

Under suitable restrictions on the multiplicities of rank $2$ flats, the 
above modular bounds are sharp, at least for the prime 
$p=3$ and for $s=1$.

\begin{theorem}[\cite{PS-betamilnor}]
\label{thm:beta3}
Let $\A$ be a hyperplane arrangement, and  
suppose $L_2(\A)$ has no flats of multiplicity $3r$, for any $r>1$. 
Then $\beta_3(\A) \ne 0$ if and only if $\A$ admits a reduced 
$3$-multinet, or, equivalently, a $3$-net. Moreover, 
$\beta_3(\A)\le 2$ and $e_3(\A)=\beta_3(\A)$. 
\end{theorem}

Putting things together, we have the following immediate corollary.  

\begin{corollary}[\cite{PS-betamilnor}]
\label{cor:triple}
Suppose $L_2(\A)$ has only flats of multiplicity $2$ and $3$.  Then 
\[
\Delta_{\A}(t)=(t-1)^{\abs{\A}-1}\cdot (t^2+t+1)^{\beta_3(\A)},
\]
where $\beta_3(\A)\in \{0, 1, 2\}$ is combinatorially determined. 
\end{corollary}

For more information on the class of `triple point' line arrangements, 
we refer to \cite{Li12, Di15, Di16, DIM}.
When multiplicity $4$ does occur, some further 
combinatorial restrictions lead to equalities in the modular 
bounds \eqref{eq:bound bis}, at the prime $p=2$ and for $s\le 2$.

\begin{theorem}[\cite{PS-betamilnor}]
\label{thm:beta4}
If $\A$ admits a $4$-net, and if 
$\beta_2(\A)\le 2$, then $e_2(\A)=e_4(\A)=\beta_2(\A)$. 
\end{theorem}

The above results, and many other computations naturally lead to 
the following conjecture.

\begin{conjecture}[\cite{PS-betamilnor}]
\label{conj:mf}
The characteristic polynomial of the degree $1$ 
algebraic monodromy for the Milnor fibration of an 
arrangement $\A$ of rank at least $3$ 
is given by the following combinatorial formula:
\begin{equation}
\label{eq:delta arr}
\Delta_{\A}(t)=(t-1)^{\abs{\A}-1} ((t+1)(t^2+1))^{\beta_2(\A)} (t^2+t+1)^{\beta_3(\A)}. 
\end{equation} 
\end{conjecture}

This conjecture has been verified for several large classes 
of arrangements, including
\begin{enumerate} 
\item all sub-arrangements of non-exceptional Coxeter arrangements, \cite{MP}; 
\item all complex reflection arrangements, \cite{MPP, Di16b, DiSt2}; 
\item certain types of complexified real arrangements, \cite{Yo, TY, BS16}.
\end{enumerate}

\section{Further topological invariants of the Milnor fiber}
\label{sect:topinv}

\subsection{Torsion in the homology of the Milnor fiber}
\label{subsec:torsion}

A long-standing question, raised by Randell and Dimca--N\'emethi 
among others, asks whether the Milnor fiber of a complex hyperplane 
arrangement can have non-trivial torsion in (integral) homology.  

As a first step towards answering this question, it was shown by 
Cohen, Denham, and Suciu  \cite{CDS03} that the first homology 
of the Milnor fiber of a multi-arrangement may have torsion.  
In recent work of Denham and Suciu \cite{DeS-plms}, these 
examples were recast in a more general framework, leading to 
hyperplane arrangements $\B$ for which $H_q(F(\B),\Z)$ has 
torsion, in some degree $q>1$. The precise result reads as follows.

\begin{theorem}[\cite{DeS-plms}]
\label{thm:polar tors}
Suppose $\A$ admits a pointed multinet, with distinguished 
hyperplane $H$.  Let $p$ be a prime dividing the multiplicity $m_H$. 
There is then a choice of multiplicities $m'$ on the deletion 
$\A' =\A\setminus \{H\}$ such that $H_q(F(\B),\Z)$ has 
$p$-torsion, where $\B$ is the arrangement obtained from 
the multi-arrangement $(\A', m')$ by a process of polarization, 
and $q=1+\abs{\set{K\in \A':  m'_K\ge 3}}$.
\end{theorem}

For instance, if $\A'$ is the deleted ${\rm B}_3$ 
arrangement mentioned in \S\ref{subsec:char var}, then 
a suitable choice of multiplicities $m'$ produces an 
arrangement $\B$ of $27$ hyperplanes in $\C^8$ such 
that $H_6(F(\B),\Z)$ has $2$-torsion. 
Nevertheless, it is still not known whether there is 
a hyperplane arrangement $\A$ (without multiplicities) 
such that $H_1(F(\A),\Z)$ has non-trivial torsion. 
For more on this topic, we refer to \cite{DeS-indam}.

\subsection{The homology of the boundary of the Milnor fiber}
\label{subsec:char poly bdF}

A detailed study of the boundary of the Milnor fiber of a non-isolated 
surface singularity was done by N\'{e}methi and Szil\'{a}rd in \cite{NS}.
When applied to arrangements in $\C^3$, their work yields the following 
result. 

\begin{theorem}[\cite{NS}]  
\label{thm:charpoly}
Let $\A$ be an arrangement of $n$ planes in $\C^3$, and 
let  $\bdF$ be the boundary of its Milnor fiber.  
The characteristic polynomial of the algebraic monodromy 
acting on $H_1(\bdF,\C)$ is equal to the product 
\[
\prod_{X\in L_2(\A)} (t-1) (t^{\gcd (\abs{\A_X},n)} -1)^{\abs{\A_X}-2}.
\]
\end{theorem}

In particular, the Betti number $b_1(\bdF)$ is determined by very 
simple combinatorial data associated to the arrangement, namely, 
the multiplicities of the rank $2$ flats. 
In general, torsion can occur in the first homology of $\bdF$.  For instance,  
as noted in \cite{NS}, 
if $\A$ is an arrangement of $4$ planes in general position 
in $\C^3$, then $H_1(\bdF, \Z)=\Z^{6} \oplus \Z_4$. 
For a generic arrangement of $n$ planes in $\C^3$, 
it is conjectured in \cite{Su-toul} that 
\begin{equation}
\label{eq:h1bdf gen}
H_1(\bdF, \Z)=\Z^{n(n-1)/2} \oplus \Z_{n}^{(n-2)(n-3)/2}.
\end{equation}

For an arbitrary arrangement $\A$ in $\C^3$, it is an open 
question whether all the torsion in $H_1(\bdF, \Z)$ 
consists of $\Z_n$-summands, where $n=\abs{\A}$. 
Likewise, it is an open question whether such torsion 
is combinatorially determined.

\subsection{Complement, boundary, and intersection lattice}
\label{subsec:homeos}

Once again, let $\A$ be an arrangement in $\C^3$, with intersection lattice 
$L(\A)$, and let $U$ be its projectivized complement.  Recall that the 
boundary manifold, $\bdU$, is a closed graph manifold, with  
underlying graph $\Gamma$ the bipartite graph whose vertices 
correspond to the lines and the intersection points of the projectivized 
line arrangement $\bar\A$.

Now suppose $\A'$ is another arrangement in $\C^3$, and that $U$ 
is homeomorphic to $U'$.  It follows that $\bdU$ is homeomorphic 
to $\bdU'$, or, equivalently (since both graph manifolds are either $S^3$  
or have positive first Betti  number),  $\bdU$ is homotopy equivalent to $\bdU'$. 
Using Waldhausen's classification of graph manifolds, 
Jiang and Yau \cite{JY93, JY98} conclude that the underlying graphs, 
$\Gamma$ and $\Gamma'$ must be isomorphic, and thus the 
corresponding intersection lattices, $L(\A)$ and $L(\A')$, must also be isomorphic.

\begin{figure}
\subfigure{%
\label{fig:f1}%
\begin{minipage}[t]{0.4\textwidth}
\setlength{\unitlength}{15pt}
\begin{picture}(5,6.1)(-3.6,-1.1)
\put(2,2){\oval(6.5,6)[t]}
\put(2,2){\oval(6.5,7)[b]}
\put(0,0){\line(1,1){4}}
\put(-1,2){\line(1,0){6}}
\put(0,4){\line(1,-1){4}}
\put(0.5,-0,5){\line(0,1){5}}
\put(3.5,-0,5){\line(0,1){5}}
\put(-1.9,3){\makebox(0,0){$\ell_0$}}
\put(0.5,-1){\makebox(0,0){$\ell_1$}}
\put(3.5,-1){\makebox(0,0){$\ell_2$}}
\put(4.5,-0.4){\makebox(0,0){$\ell_3$}}
\put(4.6,1.5){\makebox(0,0){$\ell_4$}}
\put(4.5,4.3){\makebox(0,0){$\ell_5$}}
\end{picture}\end{minipage}
}
\subfigure{%
\label{fig:f2}%
\begin{minipage}[t]{0.4\textwidth}
\setlength{\unitlength}{15pt}
\begin{picture}(5,6.1)(-3.6,-1.1)
\put(2,2){\oval(6.5,6.3)[t]}
\put(2,2){\oval(6.5,7)[b]}
\put(-1,0.5){\line(1,0){6}}
\put(-1,3.5){\line(1,0){6}}
\put(0.5,-0,5){\line(0,1){5.3}}
\put(3.5,-0,5){\line(0,1){5.3}}
\put(-0.95,-0.25){\line(1,1){5}}
\put(-1.9,3){\makebox(0,0){$\ell_0$}}
\put(0.5,-1){\makebox(0,0){$\ell_1$}}
\put(3.5,-1){\makebox(0,0){$\ell_2$}}
\put(4.5,1){\makebox(0,0){$\ell_3$}}
\put(4.5,4){\makebox(0,0){$\ell_4$}}
\put(2.1,2){\makebox(0,0){$\ell_5$}}
\end{picture}
\end{minipage}
}
\caption{The arrangements $\A$ and $\A'$}
\label{fig:falk}
\end{figure}
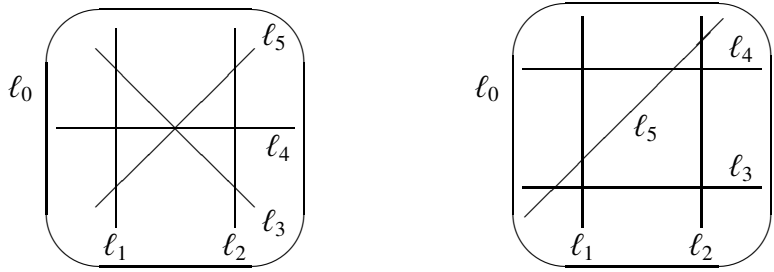

\begin{example}
\label{ex:falk arrs}
Let $\A$ and $\A'$ be the pair of arrangements in $\C^3$ 
whose projectivizations are depicted in Figure \ref{fig:falk}. 
Both $\bar\A$ and $\bar\A'$ have $2$ triple points and $9$ double points, 
yet the two intersection lattices are non-isomorphic:  the two triple 
points of $\A$ lie on a common line, whereas the two triple points 
of $\A'$ don't.  Nevertheless, as first noted by L.~Rose and H.~Terao 
in an unpublished note, the corresponding Orlik--Solomon algebras 
are isomorphic.  In fact, as shown by Falk in \cite{Fa93}, the two projective 
complements, $U$ and $U'$, are homotopy equivalent. 

Now, since $L(\A)\not\cong L(\A')$, we know from \cite{JY93, JY98} 
that the corresponding boundary manifolds, $\bdU$ and $\bdU'$, 
are not homotopy equivalent, even though $b_1(\bdU)=b_1(\bdU')=13$. 
In fact, as noted in \cite[Ex.~5.3]{CS08}, the two manifolds may be 
distinguished by their (multi-variable) Alexander polynomials:  
$\Delta_{\bdU}(t)$ has $7$ distinct factors, whereas $\Delta_{\bdU'}(t)$ has 
$8$ distinct factors.  The characteristic varieties $\VV_1(\bdU)$ and 
$\VV_1(\bdU')$ are the zero sets of these polynomials. Hence, the first 
variety consists of $7$ codimension-$1$ subtori in $(\C^*)^{13}$, 
while the second one consists of $8$ such subtori.  
This shows, once again, that $\bdU \not\simeq \bdU'$.
\end{example}

In \cite{JY93}, Jiang and Yau conjecture that the homeomorphism type 
of $U(\A)$ is determined by isomorphism type of $L(\A)$, for any arrangement 
$\A$ in $\C^3$.   Motivated by the above considerations, we propose a 
more precise conjecture.  

\begin{conjecture}
\label{conj:comb}
Let $\A$ and $\A'$ be two central arrangements in $\C^3$.  The following 
conditions are equivalent:
\begin{enumerate}
\item 
$U(\A)\cong U(\A')$.
\item 
$\bdU(\A)\simeq \bdU(\A')$.
\item 
$\Delta_{\bdU(\A)}(t)=\Delta_{\bdU(\A')}(t)$.
\item 
$\Gamma(\A) \cong \Gamma (\A')$.
\item 
$L(\A)\cong L(\A')$.  
\end{enumerate}
\end{conjecture}

\subsection{Milnor fiber and intersection lattice}
\label{subsec:falk mf}

We now show that there are invariants which can 
tell apart homologically equivalent Milnor fibers of arrangements.  

\begin{example}
\label{ex:falk mf}
Let $\A$ and $\A'$ be the two arrangements from Example \ref{ex:falk arrs}, 
and let $F$ and $F'$ be the corresponding Milnor fibers.  It is readily seen 
that neither of the two arrangements supports an essential multinet.  
Since both $\bar\A$ and $\bar\A'$ have only double and triple points, 
Corollary \ref{cor:triple} shows that, in both cases, the characteristic 
polynomial of the algebraic monodromy 
acting on the Milnor fiber is $(t-1)^5$.  
By the same token, Theorem \ref{thm:charpoly} shows that, in both cases, 
the characteristic polynomial of the algebraic monodromy 
acting on the boundary of the Milnor fiber is  
$(t-1)^{13}(t^2+t+1)^2$.

It can also be verified that $H_1(F,\Z)=H_1(F',\Z)=\Z^5$.  
Nevertheless, the two Milnor fibers are {\em not}\/ homotopy equivalent.  
In fact, we claim that $\pi_1(F)\not\cong \pi_1(F')$.  
To establish this claim, we consider the characteristic varieties 
of $F$ and $F'$, lying in the character torus $(\C^*)^5$.  
A computation shows that
\begin{align*}
\VV_1(F)&=\{  t_1=t_4=t_5=1\}\cup 
\{t_3^3t_5^{-1} =
t_3^6t_4^{-1} =
t_2 t_3^{-1} =1
\},\\
\VV_2(F)&=\{\bo , (1,\omega,\omega, 1, 1), (1,\omega^{2},\omega^{2}, 1, 1)\},
\\
\intertext{where $\omega=\exp(2\pi \ii /3)$, while}
\VV_1(F')&=\{ t_1 t_4= t_1 t_5 = t_3 t_5^{-3}=1 \}\cup 
\{t_2 t_4 = t_3 t_5^{2} = t_4 t_5= 1
\},\\
\VV_2(F')&=\{\bo\}.
\end{align*}

Note that the two, $2$-dimensional components 
of $\VV_1(F)$ meet at the three characters of order $3$ comprising $\VV_2(F)$. 
The variety $\VV_1(F')$ also consists of two, $2$-dimensional subtori, 
but these subtori only meet at the origin, which is the 
only point comprising $\VV_2(F')$.  
\end{example}

In view of these considerations, we conclude with a (rather optimistic) 
conjecture, which can be viewed as a Milnor fiber analogue of
Conjecture~\ref{conj:comb}.

\begin{conjecture}
\label{conj:mf bdmf}
Let $\A$ and $\A'$ be two central arrangements in $\C^3$.  The following 
conditions are equivalent:
\begin{enumerate}
\item
$F(\A)\cong F(\A')$.
\item 
$\bdF(\A)\simeq \bdF(\A')$.
\item 
$L(\A)\cong L(\A')$.  
\end{enumerate}
\end{conjecture}

\section*{Acknowledgement}
Most of this work was done while the author visited the Institute of 
Mathematics of the Romanian Academy in June, 2015 and 
in June, 2016.  He thanks IMAR for its hospitality, 
support, and excellent research atmosphere. 
He also thanks the organizers of the Eighth 
Congress of Romanian Mathematicians, 
held in Ia\c{s}i in June 2015 
for the opportunity to present a preliminary 
version of this work.  


\newcommand{\arxiv}[1]
{\href{http://arxiv.org/abs/#1}{arXiv:#1}}

\newcommand{\arx}[1]
{\texttt{\href{http://arxiv.org/abs/#1}{arXiv:}}
\texttt{\href{http://arxiv.org/abs/#1}{#1}}}

\newcommand{\doi}[1]
{\texttt{\href{http://dx.doi.org/#1}{doi:#1}}}

\renewcommand{\MR}[1]
{\href{http://www.ams.org/mathscinet-getitem?mr=#1}{MR#1}}

\newcommand{\MRh}[2]
{\href{http://www.ams.org/mathscinet-getitem?mr=#1}{MR#1 (#2)}}

\end{document}